\documentclass[reqno]{amsart}

\input xy
\xyoption{all}
\usepackage{accents}
\usepackage{epsfig}
\usepackage{xcolor}
\usepackage{amsthm}
\usepackage{amssymb}
\usepackage{amsmath}
\usepackage{amscd}
\usepackage{amsopn}
\usepackage{graphicx}
\usepackage{tikz-cd} 
\usepackage{xspace}

\usepackage{hhline}
\usepackage{easybmat}
\usepackage{caption}   
\usepackage{relsize}

\usepackage{url}
\usepackage{enumitem, hyperref}\hypersetup{colorlinks}

\usepackage{color} 

\definecolor{darkred}{rgb}{1,0,0} 
\definecolor{darkgreen}{rgb}{0,0.8,0}
\definecolor{darkblue}{rgb}{0,0,1}

\hypersetup{colorlinks,
linkcolor=darkblue,
filecolor=darkgreen,
urlcolor=darkred,
citecolor=darkgreen}

\makeatletter
\def\reflb#1#2{\begingroup
    #2%
    \def\@currentlabel{#2}%
    \phantomsection\label{#1}\endgroup
}
\makeatother

\numberwithin{equation}{section}
\newtheorem {Theorem}{Theorem}
\numberwithin{Theorem}{section}

\newtheorem {Lemma}[Theorem]    {Lemma}

\newtheorem {Proposition}[Theorem]{Proposition}

\theoremstyle{definition}
\newtheorem{Definition}[Theorem]{Definition}
\theoremstyle{remark}
\newtheorem{Remark}[Theorem]{Remark}
\newtheorem{Example}[Theorem]{Example}

\def    \eps    {\epsilon}

\newcommand{\CA}{{\mathcal A}}

\newcommand{\CS}{{\mathcal S}}

\newcommand{\Ham}{{\mathit{Ham}}}
\newcommand{\tHam}{{\widetilde{\mathit{Ham}}}}

\newcommand{\id}{{\mathit id}}

\newcommand{\bPP}{\bar{\mathcal P}}

\def    \F      {{\mathbb F}}

\def    \R      {{\mathbb R}}

\def    \Z      {{\mathbb Z}}
\def    \N      {{\mathbb N}}

\def    \CP     {{\mathbb C}{\mathbb P}}
\def    \RP     {{\mathbb R}{\mathbb P}}

\def    \12    {{\frac{1}{2}}}

\def    \tSp     {\operatorname{\widetilde{Sp}}}

\def    \HF     {\operatorname{HF}}

\def    \HQ     {\operatorname{HQ}}

\def    \H     {\operatorname{H}}
\def    \Co    {\operatorname{C}}

\def    \CMo     {\operatorname{CM}}

\def    \CF     {\operatorname{CF}}

\def    \bPP     {\bar{\mathcal{P}}}
\def    \bx     {\bar{x}}
\def    \by     {\bar{y}}
\def    \bz     {\bar{z}}

\def  \hmu      {\operatorname{\hat{\mu}}}

\def \Fl   {\scriptscriptstyle{Fl}}

\def \qq   {\mathrm{q}}
\def \hh   {\mathrm{h}}

\newcommand \Sq {\operatorname{Sq}}

\newcommand \QS {\mathcal{QS}}
 \newcommand \PS {\mathcal{PS}}

\newcommand \QC {\operatorname{QC}}
\newcommand \eq {\scriptscriptstyle{eq}}

\begin{document}


\setlength{\smallskipamount}{6pt}
\setlength{\medskipamount}{10pt}
\setlength{\bigskipamount}{16pt}





\title[Pseudo-Rotations and quantum Steenrod squares]{From
  Pseudo-Rotations to holomorphic curves via quantum Steenrod squares
}

\author[Erman \c C\. inel\. i]{Erman \c C\. inel\. i}
\author[Viktor Ginzburg]{Viktor L. Ginzburg}
\author[Ba\c sak G\"urel]{Ba\c sak Z. G\"urel}

\address{E\c C and VG: Department of Mathematics, UC Santa Cruz, Santa
  Cruz, CA 95064, USA} \email{scineli@ucsc.edu}
\email{ginzburg@ucsc.edu}

\address{BG: Department of Mathematics, University of Central Florida,
  Orlando, FL 32816, USA} \email{basak.gurel@ucf.edu}

\subjclass[2010]{53D40, 37J10, 37J45} 

\keywords{Pseudo-rotations, periodic orbits, Hamiltonian
  diffeomorphisms, Steenrod squares, holomorphic curves}

\date{\today} 

\thanks{The work is partially supported by NSF CAREER award
  DMS-1454342 (BG) and by Simons Collaboration Grant 581382 (VG)}


\begin{abstract}
  In the context of symplectic dynamics, pseudo-rotations are
  Hamiltonian diffeomorphisms with finite and minimal possible number
  of periodic orbits. These maps are of interest in both dynamics and
  symplectic topology. We show that a closed, monotone symplectic
  manifold, which admits a non-degenerate pseudo-rotation, must have a
  deformed quantum Steenrod square of the top degree element and hence
  non-trivial holomorphic spheres. This result (partially) generalizes
  a recent work by Shelukhin and complements the results by the
  authors on non-vanishing Gromov--Witten invariants of manifolds
  admitting pseudo-rotations.
\end{abstract}

\maketitle

\tableofcontents

\section{Introduction and main results}

\subsection{Introduction}
\label{sec:intro}
In this paper we continue investigating connections between
pseudo-rotations and symplectic topology of the underlying manifold.
We show that a closed monotone symplectic manifold, which admits a
non-degenerate Hamiltonian pseudo-rotation, must have a deformed
quantum Steenrod square of the top degree element and, as a
consequence, non-trivial holomorphic spheres. This result (partially)
generalizes a recent work by Shelukhin, \cite{Sh}, and complements the
results from \cite{CGG}.

In the context of symplectic dynamics, pseudo-rotations are
Hamiltonian diffeomorphisms with finite and minimal possible number of
periodic orbits. Such diffeomorphisms of the disk or the 2-sphere
occupy a distinguished place in low-dimensional dynamical systems
theory (see, e.g., \cite{AK, FaHe, FK, FM, LCY} and references
therein). In all dimensions, pseudo-rotations can have extremely
interesting dynamics.  For instance, some manifolds (e.g., $\CP^n$)
admit pseudo-rotations with finite number of ergodic measures; see
\cite{AK, FK, LRS}. Somewhat surprisingly, as has been recently
discovered, symplectic topological techniques can be effectively used
to study the dynamics of pseudo-rotations in dimension two (see
\cite{Br:Annals, Br, BH}) and also in higher dimensions; see
\cite{GG:PR, GG:PRvR}.

Many (in all likelihood, most) closed symplectic manifolds do not
admit pseudo-rotations or even Hamiltonian diffeomorphisms with
finitely many periodic orbits as the Conley conjecture asserts; see
\cite{GG:survey}. The state of the art result is that a closed
symplectic manifold $(M,\omega)$ does not admit such a Hamiltonian
diffeomorphism unless there exists a class $A\in\pi_2(M)$ such that
$\left<[\omega],A\right>>0$ and $\left<c_1(TM),A\right>>0$, \cite{Ci,
  GG:Rev}. For instance, this criterion rules out all cases where the
class $[\omega]$ or $c_1(TM)$ is aspherical and also negative monotone
symplectic manifolds. While this theorem is satisfactory on the
topological level, one can expect the failure of the Conley conjecture
for $M$ to also have strong consequences on the symplectic topological
level. For instance, the so-called Chance--McDuff conjecture, inspired
by \cite{McD}, claims that a symplectic manifold admitting a
Hamiltonian diffeomorphism with finitely many periodic orbits must
have non-vanishing Gromov--Witten invariants.

There are a few slightly different working definitions of
pseudo-rotations, all reflecting the same condition that the map must
have the least possible number of periodic orbits; \cite{CGG, Sh:HZ,
  Sh}.  In this paper, we define a pseudo-rotation as a Hamiltonian
diffeomorphism $\varphi$ such that all iterates $\varphi^k$, $k\in\N$,
are non-degenerate and the Floer differential (over $\F_2=\Z_2$)
vanishes for all $\varphi^k$; see Definition \ref{def:PR} and Remark
\ref{rmk:generalization}. (Here and throughout all cohomology groups
are with $\F_2$-coefficients.) As of this writing, all known
Hamiltonian diffeomorphisms with finitely many periodic orbits are
pseudo-rotations in this sense and these two classes might well
coincide; see \cite{Sh:HZ} for some results in this direction.

Several variants of the Chance--McDuff conjecture have been
established recently for pseudo-rotations. In \cite{CGG}, it was
proved that a weakly monotone symplectic manifold with minimal Chern
number $N>1$, admitting a pseudo-rotation $\varphi$, must have
deformed quantum product and, in particular, non-vanishing
Gromov--Witten invariants, under certain additional index assumptions
on $\varphi$. These extra assumptions appear to be satisfied for most,
but certainly not all, pseudo-rotations. Simultaneously and
independently, in \cite{Sh}, along with some other results the quantum
Steenrod square of the top degree class was shown to be deformed for
monotone symplectic manifolds $M^{2n}$ with Poincar\'e duality
property (e.g., when $N\geq n+1$), admitting pseudo-rotations. (These
pseudo-rotations need not be non-degenerate.)

The quantum Steenrod square is a symplectic topological invariant
introduced in \cite{Se} and then studied in \cite{Wi1, Wi2}; see also
\cite{Be, BC, Fu, He, SS, ShZ} for relevant work. It is a cohomology
operation $\QS$ on the quantum cohomology $\HQ^*(M)$, which is a
deformation of the standard Steenrod square. In other words,
$\QS(\alpha)=\Sq(\alpha)+O(\qq)$, where $\qq$ is the generator of the
Novikov ring and $\alpha\in \H^*(M)$.  Roughly speaking, a deformed
quantum Steenrod square, just as a deformed quantum product $*$,
detects certain holomorphic spheres in $M$, but in general these
spheres need not be related to Gromov--Witten invariants. Furthermore,
$\QS$ can also be viewed as a deformation of the standard quantum
product (with respect to a different parameter $\hh$) in the same
sense as $\Sq$ can be thought of as a deformation of the cup-product.
On the Floer cohomology side, $\QS$ is closely related to another
quantum cohomology operation also introduced in \cite{Se}, the
equivariant pair-of-pants product $\mathlarger{\wp}$, which plays a
crucial role in our proof. We will briefly discuss both of these
operations in Section \ref{sec:prelim}.

Let $\varpi$ be the generator of the top degree cohomology group
$\H^{2n}(M^{2n})$. Our main result, Theorem \ref{thm:main}, is that
$\QS(\varpi)$ is different from $\Sq(\varpi)=\hh^{2n}\varpi$ whenever
$M$ is monotone and admits a pseudo-rotation. Here we treat the
Steenrod square $\Sq$ as a degree doubling map
$$
\Sq\colon \H^*(M)\to \H^*(M)[[\hh]], \quad
\Sq(\alpha)=\sum_{i=0}^{|\alpha|} \hh^{|\alpha|-i}\Sq^i(\alpha),
$$
where $|\hh|=1$ and $\Sq^i(\alpha)$ is the $i$-th standard Steenrod
square of $\alpha\in \H^*(M)$, and $\H^*(M)[[\hh]]$ is the space of
formal power series with coefficients in $\H^*(M)$, \cite{Se, Wi1,
  Wi2}. As a consequence, there is a non-trivial holomorphic sphere
through every point of $M$.

Apart from the non-degeneracy condition, this result generalizes the
main theorem from \cite{Sh} by eliminating the Poincar\'e duality
requirement; cf.\ Remark \ref{rmk:generalization}. On the other hand,
it is difficult to compare Theorem \ref{thm:main} and the results from
\cite{CGG} detecting a deformed quantum product; for these theorems
hold under different conditions and provide different symplectic
topological information. Note however that the statement that $\QS$ is
deformed is obviously much weaker than that its 0-th order term in
$\hh$, the quantum square, is deformed, i.e.,
$\alpha *\alpha \neq \alpha\cup\alpha$ for some $\alpha\in \H^*(M)$.

The proof of Theorem \ref{thm:main} hinges on the same idea as the
argument in \cite{CGG}, although the latter proof is considerably more
involved. In both cases, a non-trivial deformation comes roughly
speaking from constant (to be more precise, zero energy) pair-of-pants
solutions of the Floer equation: equivariant in the present case and
standard for the quantum product.

For the standard pair-of-pants product, a zero energy curve is easily
seen to be automatically regular provided that the Conley--Zehnder
indices allow this. Namely, consider the cohomology pair-of-pants
product of iterated capped periodic orbits
$\bx^{k_1}*\ldots*\bx^{k_r}$.  Then the least action term in this
product is $\bx^{k}$ with $k=k_1+\ldots+k_r$, i.e.,
$$
\bx^{k_1}*\ldots*\bx^{k_r}=\bx^k+\ldots ,
$$
where the dots stand for higher action terms, if and only if $\bx^k$
has the ``right'' Conley--Zehnder index.  Explicitly, with our
conventions, the latter index condition is that
\[
\mu\big(\bx^k\big)=\mu\big(\bx^{k_1}\big)+\ldots+\mu\big(\bx^{k_r}\big)
+(r-1)n,
\]
and the main difficulty in the proof in \cite{CGG} is to guarantee
that this requirement is satisfied for some orbit $\bx$ and that the
resulting product is different from the cup product.

On the other hand, for the equivariant pair-of-pants product, $\bx^2$
(or, to be more precise, its product with a suitable power of $\hh$)
is always, without any index requirement, the least action term in
$\mathlarger{\wp}(\bx \otimes \bx)$, although now this is a
non-trivial fact proved in \cite{Se}; see also \cite{ShZ}.  This is
sufficient to show that the quantum Steenrod square of $\varpi$ is
deformed whenever $M$ admits a pseudo-rotation, by using
simultaneously the action and $\hh$-adic filtrations of the
equivariant Floer cohomology.

\subsection{Main result: from pseudo-rotations to the Steenrod
  square}
\label{sec:results}
While, as has been mentioned above, there are several slightly
different ways to define a pseudo-rotation, for our purposes it is
convenient to adopt the following definition.

\begin{Definition}
  \label{def:PR}
  A Hamiltonian diffeomorphism $\varphi$ of a closed weakly monotone
  symplectic manifold is called a \emph{pseudo-rotation} if all
  iterates $\varphi^k$, $k\in\N$, are non-degenerate and the Floer
  differential over $\F_2$ vanishes for all $\varphi^k$.
\end{Definition}

We note that while the Floer differential depends on an auxiliary
structure (an almost complex structure), its vanishing is
well-defined, i.e., independent of this structure.

We are now in a position to state the main result of this paper.  Let
$\QS$ be the quantum Steenrod square; see Section \ref{sec:prelim} for
a brief discussion and further references.

\begin{Theorem}
  \label{thm:main}
  Assume that a closed monotone symplectic manifold $(M^{2n}, \omega)$
  admits a pseudo-rotation. Then the quantum Steenrod square $\QS$ of
  the top degree cohomology class $\varpi\in \H^{2n}(M;\F_2)$ is
  deformed: $\QS(\varpi)\neq \hh^{2n}\varpi$.
\end{Theorem}

\begin{Remark}
  \label{rmk:generalization}
  Although we deliberately limited our attention to a relatively
  narrow class of manifolds and pseudo-rotations, we expect the
  theorem to have several generalizations accessible by the same
  method with relatively minor modifications. First of all, the
  theorem should hold for a slightly broader class of pseudo-rotations
  considered in \cite{Sh:HZ, Sh}, allowing for some degenerations.
  \footnote{In fact, after this work had been completed we have
    learned of \cite{Sh:new} where a similar theorem has been proved
    in such a setting by a different method.} Next, one might be able
  to replace the assumption that $M$ is monotone by the condition that
  it is weakly monotone; for one can expect the constructions of the
  equivariant pair-of-pants product $\mathlarger{\wp}$ from \cite{Se}
  and of the quantum Steenrod square $\QS$ to extend to this setting
  with some modifications; see \cite{SW}. Finally, one might also be
  able to extend Theorem \ref{thm:main} to the quantum Steenrod $\Z_p$
  cohomology operations (see \cite{ShZ}) although some details of the
  underlying machinery are yet to be finalized.
  
\end{Remark}

\noindent{\bf Acknowledgements.} The authors are grateful to Egor
Shelukhin, Nicholas Wilkins and the referee for useful remarks and
discussions.
  
\section{Conventions and notation}
\label{sec:conv}
For the reader's convenience we set here our conventions and notation
and briefly recall some basic definitions. The reader may want to
consult this section only as needed.

Throughout this paper, the underlying symplectic manifold $(M,\omega)$
is assumed to be closed and (strictly) monotone, i.e.,
$[\omega]|_{\pi_2 (M)}=\lambda c_1(TM)|_{\pi_2(M)}\neq 0$ for some
$\lambda>0$. The \emph{minimal Chern number} of $M$ is the positive
generator $N$ of the subgroup
$\left<c_1(TM),\pi_2(M)\right>\subset \Z$.

A \emph{Hamiltonian diffeomorphism} $\varphi=\varphi_H=\varphi_H^1$ is
the time-one map of the time-dependent flow $\varphi^t=\varphi_H^t$ of
a 1-periodic in time Hamiltonian $H\colon S^1\times M\to\R$, where
$S^1=\R/\Z$.  The Hamiltonian vector field $X_H$ of $H$ is defined by
$i_{X_H}\omega=-dH$. Such time-one maps form the group
$\Ham(M,\omega)$ of Hamiltonian diffeomorphisms of $M$. In what
follows, it will be convenient to view Hamiltonian diffeomorphisms as
elements of the universal covering $\tHam(M,\omega)$.

Let $x\colon S^1\to M$ be a contractible loop. A \emph{capping} of $x$
is an equivalence class of maps $A\colon D^2\to M$ such that
$A|_{S^1}=x$. Two cappings $A$ and $A'$ of $x$ are equivalent if the
integral of $\omega$ (or of $c_1(TM)$ since $M$ is strictly monotone)
over the sphere obtained by attaching $A$ to $A'$ is equal to zero. A
capped closed curve $\bar{x}$ is, by definition, a closed curve $x$
equipped with an equivalence class of cappings, and the presence of
capping is indicated by a bar.

The action of a Hamiltonian $H$ on a capped closed curve
$\bar{x}=(x,A)$ is
$$
\CA_H(\bar{x})=-\int_A\omega+\int_{S^1} H_t(x(t))\,dt.
$$
The space of capped closed curves is a covering space of the space of
contractible loops, and the critical points of $\CA_H$ on this space
are exactly the capped 1-periodic orbits of $X_H$.

The $k$-periodic \emph{points} of $\varphi$ are in one-to-one
correspondence with the $k$-periodic \emph{orbits} of $H$, i.e., of
the time-dependent flow $\varphi^t$. Recall also that a $k$-periodic
orbit of $H$ is called \emph{simple} if it is not iterated.  A
$k$-periodic orbit $x$ of $H$ is said to be \emph{non-degenerate} if
the linearized return map $d\varphi^k \colon T_{x(0)}M \to T_{x(0)}M$
has no eigenvalues equal to one. We call $x$ \emph{strongly
  non-degenerate} if all iterates $x^k$ are non-degenerate. A
Hamiltonian $H$ is non-degenerate if all its 1-periodic orbits are
non-degenerate. \emph{In what follows we always assume that $H$ is
  strongly non-degenerate, i.e., all periodic orbits of $H$ (of all
  periods) are non-degenerate.} We denote the collection of capped
$k$-periodic orbits of $H$ by $\bPP_k(H)$ or $\bPP_k(\varphi)$.

Let $\bar{x}$ be a non-degenerate capped periodic orbit.  The
\emph{Conley--Zehnder index} $\mu(\bar{x})\in\Z$ is defined, up to a
sign, as in \cite{Sa,SZ}. In this paper, we normalize $\mu$ so that
$\mu(\bar{x})=n$ when $x$ is a non-degenerate maximum (with trivial
capping) of an autonomous Hamiltonian with small Hessian. The
\emph{mean index} $\hmu(\bx)\in\R$ measures, roughly speaking, the
total angle swept by certain (Krein--positive) unit eigenvalues of the
linearized flow $d\varphi^t|_{\bx}$ with respect to the trivialization
associated with the capping; see \cite{Lo,SZ}. The mean index is
defined even when $x$ is degenerate, and we always have the inequality
$\big|\hmu(\bx)-\mu(\bx)\big|\leq n$. Moreover, if $x$ is
non-degenerate, the inequality is strict:
\begin{equation}
\label{eq:mean-cz}
\big|\hmu(\bx)-\mu(\bx)\big|<n .
\end{equation}
The mean index is homogeneous with respect to iteration:
$\hmu\big(\bx^k\big)=k\hmu(\bx)$. (The capping of $\bx^k$ is obtained
from the capping of $\bx$ by taking its $k$-fold cover branched at the
origin.)

Fixing an almost complex structure, which will be suppressed in the
notation, we denote by $\left(\CF^*(\varphi), d_{\Fl}\right)$ and
$\HF^*(\varphi)$ the Floer complex and cohomology of $\varphi$ over
$\F_2=\Z_2$; see, e.g., \cite{MS, Sa}. (As has been mentioned in the
introduction, here all complexes and cohomology groups are over
$\F_2$.) The complex $\CF^*(\varphi)$ is generated by the capped
1-periodic orbits $\bx$ of $H$, graded by the Conley--Zehnder index,
and filtered by the action. The differential $d_{\Fl}$ is the upward
Floer differential: it increases the action and also the index by
one. With these conventions, we have the canonical isomorphism
\begin{equation}
\label{eq:Fl-qt}
\Phi\colon
\HQ^*(M)\stackrel{\cong}{\longrightarrow} \HF^*(\varphi)[n],
\end{equation}
where $\HQ^*(M)$ is the quantum homology of $M$; see, e.g.,
\cite{Sa,MS} and references therein. Depending on the context, $\Phi$
is the PSS-isomorphism or the continuation map or a combination of the
two.

For instance, assume that $H$ is $C^2$-small and autonomous (i.e.,
independent of $t$), and has a unique maximum and a unique
minimum. Then the top degree cohomology class
$\varpi\in \H^{2n}(M)\subset \HQ^{2n}(M)$ corresponds to the maximum
of $H$, which has degree $n$ in $\HF^*(\varphi)$; the unit
$1\in \HQ^0(M)$ corresponds to the minimum of $H$ which has degree
$-n$ in $\HF^*(\varphi)$. We denote by $|\alpha|$ the degree of
$\alpha\in \HQ^*(M)$ or $\alpha\in\HF^*(\varphi)$.

The cohomology groups $\HQ^*(M)$ and $\HF^*(\varphi)$ and the complex
$\CF^*(\varphi)$ are modules over a Novikov ring $\Lambda$, and
$\HQ^*(M)\cong \H^*(M)\otimes \Lambda$ (as a module). There are
several choices of $\Lambda$; see, e.g., \cite{MS}. For our purposes,
it is convenient to take the field of Laurent series $\F_2((\qq))$
with $|\qq|=2N$ as $\Lambda$.  Then $\Lambda$ naturally acts on
$\CF^*(\varphi)$ by recapping, and multiplication by $\qq$ corresponds
to the recapping by $A\in\pi_2(M)$ with $\left<c_1(TM),A\right>=N$.

When $\varphi$ is a pseudo-rotation (or, more generally, if
$d_{\Fl}=0$), the isomorphism \eqref{eq:Fl-qt} turns into the natural
identification
$$
\HQ^*(M)[-n]\cong \HF^*(\varphi) \cong \CF^*(\varphi).
$$
Since any iterate $\varphi^k$ is also a pseudo-rotation, we have
$$
\HQ^*(M)[-n] \cong \HF^*\big(\varphi^k\big) 
\cong   \CF^*\big(\varphi^k\big).
$$
It is worth emphasizing that the resulting isomorphism between
$\CF^*(\varphi)=\HF^*(\varphi)$ and
$\CF^*\big(\varphi^k\big)=\HF^*\big(\varphi^k\big)$, which is given by
the continuation map, is usually different from the iteration map
$\bx\mapsto \bx^k$. For instance, unless $M$ is aspherical the
iteration map is not onto and, in general,
$\mu(\bx)\neq \mu\big(\bx^k\big)$ even in the aspherical case.

The quantum homology $\HQ^*(M)$ carries the \emph{quantum product},
denoted here by $*$, which makes it into a graded-commutative algebra
over $\Lambda$ with unit $1$. This product is a deformation (in $\qq$)
of the cup product: $\alpha*\beta=\alpha\cup\beta+O(\qq)$. For
instance, $\alpha_1*\alpha_n=\qq\cdot 1$ in $\HQ^*(\CP^n)$, where
$\alpha_l$ stands for the generator (i.e., the only non-zero element)
of $\H^{2l}(\CP^n)$. In Floer cohomology, the quantum product
corresponds to the so-called \emph{pair-of-pants product}
$$
\HF^*(\varphi)\otimes \HF^*(\varphi)\to \HF^*\big(\varphi^2\big)[n],
$$
which we also denote by $*$. We emphasize that with our conventions
$|\alpha * \beta|=|\alpha|+|\beta|+n$ in Floer cohomology. When
\eqref{eq:Fl-qt} is applied to $\varphi$ and $\varphi^2$, the
pair-of-pants product turns into the quantum product:
$$
\HQ^*(M)\otimes\HQ^*(M)\cong \HF^*(\varphi)\otimes \HF^*(\varphi)\to
\HF^*\big(\varphi^2\big)\cong \HQ^*(M).
$$
Here, for the sake of simplicity, we suppressed the shifts of degree
as we often will in what follows.

\section{Preliminaries}
\label{sec:prelim}
Drawing heavily from \cite{Se} and also \cite{ShZ,Wi1,Wi2}, we now
recall the construction of the equivariant pair-of-pants product and
its relation to the quantum Steenrod square in the case where the
ambient manifold $M$ is closed and monotone. Then we will take a
closer look at the effect of the additional condition that the
Hamiltonian diffeomorphism $\varphi$ is a pseudo-rotation.

\subsection{Equivariant pair-of-pants product} Let us start by briefly
outlining the definition of the \emph{equivariant pair-of-pants
  product}, closely following \cite{Se}. In this (sub)section we make
little or no use of the fact that the underlying manifold $M$ is
compact and that we are interested in the global, rather than
filtered, Floer cohomology; most of the discussion carries over to the
filtered setting verbatim. (Throughout the paper, by the filtered
Floer cohomology we mean the homology of the Floer complex restricted
to some action interval and the global Floer cohomology stands for the
homology of the entire Floer complex.)  These extra conditions afford
some immediate simplifications which we will spell out in Section
\ref{sec:QS} based on \cite{Wi1, Wi2}. It is worth pointing out that
the underlying assumptions in \cite{Se} are somewhat different from
ours: there the ambient manifold $M$ is assumed to be exact rather
than closed and monotone, but the definitions and results translate
readily to our framework; see \cite{ShZ, Wi1, Wi2}.

The target space of the equivariant pair-of-pants product is the
$\Z_2$-equivariant cohomology $\HF_{\eq}^*\big(\varphi^2\big)$. This
is the homology of the complex $\CF_{\eq}^*\big(\varphi^2\big)$, which
is the graded $\F_2$-vector space
\begin{equation}
  \label{eq:CFeq}
    \CF_{\eq}^*\big(\varphi^2\big):=
    \CF^*\big(\varphi^2\big)[[\hh]]=
    \CF^*\big(\varphi^2\big)\otimes_{\Lambda}
    \Lambda[[\hh]]
\end{equation}
with $|\hh|=1$, where here and throughout the paper $[[\hh]]$
indicates the space of formal power series, equipped with the
differential
\[
d_{\eq}=d_{\Fl}+\sum_{j=1}^{\infty}\hh^j d_j
\]
emulating the construction of the $\Z_2$-equivariant cohomology via
Morse theory on the Borel quotient; see \cite[Sect.\ 4.2]{Se} for
details. (Here $\F_2[[\hh]]\cong \H^*(\RP^\infty;\F_2)$.) This complex
and hence the homology carry two increasing filtrations: the action
filtration and the $\hh$-adic filtration. Moreover,
$\CF_{\eq}^*\big(\varphi^2\big)$ is a $\Lambda$-module with the action
of $\qq$ given by recapping as in the ordinary Floer complex. The
differential $d_{\eq}$ is $\Lambda$-linear, and hence
$\HF_{\eq}^*\big(\varphi^2\big)$ is also a $\Lambda$-module.

\begin{Remark}
  In this connection we point out that there are two slightly
  different constructions of the equivariant Floer cohomology for
  Hamiltonians with symmetry, e.g., the $S^1$-symmetry for autonomous
  Hamiltonians and the $\Z_k$-symmetry for $k$-iterated Hamiltonian
  diffeomorphisms. The first construction uses a parametrized
  perturbation of the original Hamiltonian and the action functional;
  see \cite{BO,Vi} and also \cite{GG:convex}. This is a Floer
  theoretic analogue of taking a Morse perturbation of the pull-back
  (which is Morse--Bott) of the original Morse function to the Borel
  quotient. In the second construction one keeps the Hamiltonian and
  the action functional unchanged, but uses a parametrized almost
  complex structure and continuation maps along the gradient lines of
  an auxiliary Morse function on the classifying space to define the
  differential; see \cite{Hu,Se,SS}. This approach results in a
  complex and cohomology \emph{a priori} better behaving with respect
  to the action filtration. This is the complex considered here. (The
  difference becomes apparent in the context of the filtered Leray
  spectral sequence converging to the equivariant cohomology and
  associated with the $\hh$-adic filtration: it is not even clear how
  to define the $\hh$-adic filtration in the framework of the first
  construction without additional assumptions on the perturbation; see
  \cite{BO}.) We also note that for monotone symplectic manifolds
  there are other choices in the definition of
  $\CF_{\eq}^*\big(\varphi^2\big)$, e.g.,
  $\CF^*\big(\varphi^2\big)[\hh]$, which we briefly touch upon in
  Remark \ref{rmk:CFeq}.
\end{Remark}

The domain of the equivariant pair-of-pants product map is the
ordinary group cohomology of $\Z_2$ with coefficients in the complex
$\CF^*(\varphi)\otimes_{\F_2} \CF^*(\varphi)$ with the $\Z_2$-action
given by the involution $\iota$ interchanging the two factors. In
other words, this is the cohomology of the complex
\begin{equation*}
  \Co^*\big(\Z_2; \CF^*(\varphi) \otimes_{\F_2}
  \CF^*(\varphi)\big)
  :=\big ( \CF^*(\varphi)\otimes_{\F_2}
  \CF^*(\varphi) \big )[[\hh]]
\end{equation*}
equipped with the differential $d_{\Z_2}=d_{\Fl}+\hh(\id+\iota)$,
where the first term stands for the differential induced by $d_{\Fl}$
on the tensor product. This complex also carries two increasing
filtrations: the action filtration and the $\hh$-adic filtration.  We
note that the complex is ``unaware'' of simple 2-periodic orbits of
$\varphi$.

The \emph{equivariant pair-of-pants product} is the
$\F_2[[\hh]]$-linear map
\[
  \mathlarger{\wp}\colon \H^*\big(\Z_2; \CF^*(\varphi)
  \otimes_{\F_2} \CF^*(\varphi)\big)\to
  \HF_{\eq}^*\big(\varphi^2\big)
\]
induced by the chain map
\[
  \Co^*\big(\Z_2; \CF^*(\varphi) \otimes_{\F_2}
  \CF^*(\varphi)\big)\to \CF_{\eq}^*\big(\varphi^2\big)
\]
constructed in \cite{Se}. This product is a deformation in $\hh$ of
the pair-of-pants product, i.e., on the chain level the equivariant
pair-of-pants product $\mathlarger{\wp}(c_1 \otimes c_2)$ of $c_1$ and
$c_2$ in $\CF_{eq}^*(\varphi)$ has the form $c_1 *
c_2+O(\hh)$. Furthermore, on the chain level, $\wp$ is bi-linear,
i.e.,
$\wp(\qq c_1\otimes c_2)=\wp(c_1 \otimes \qq c_2)=\qq \wp(c_1 \otimes
c_2)$. Note, however, that
$\qq c_1 \otimes c_2 \neq c_1\otimes \qq c_2$, since the tensor
product is taken over $\F_2$.
  
To get a better understanding of how the map $\mathlarger{\wp}$ works,
note first that the graded space $\CF^*\big(\varphi^2\big)$ has a
canonical involution $\iota'$ given by the shift of time
$x(t)\mapsto x(t+1)$; for the generators of this space are the
2-periodic orbits of $\varphi$. In general, this linear map, extended
to $\CF^*\big(\varphi^2\big)[[\hh]]$, does not commute with $d_{\Fl}$
unless there is a regular 1-periodic (rather than 2-periodic) almost
complex structure. However, when this is the case, one can replace the
complex $\CF_{\eq}^*\big(\varphi^2\big)$ by the former complex with
the differential $d_{\Fl}+\hh(\id+\iota')$. Then $\mathlarger{\wp}$ is
a deformation of the map induced by the pair-of-pants product map
$\CF(\varphi)\otimes_{\F_2} \CF(\varphi)\to \CF\big(\varphi^2\big)$ of
the ``coefficient'' complexes.

\begin{Remark}
  Recall that even when $\iota'$ does not commute with $d_{\Fl}$, it
  becomes an isomorphism of complexes when the target is equipped with
  the Floer differential associated with the time-shifted almost
  complex structure $J'_t=J_{t+1}$. Then, once composed with the
  continuation map, $\iota'$ induces an involution of
  $\HF^*\big(\varphi^2\big)$. In our setting, the global Floer
  cohomology $\HF^*\big(\varphi^2\big)$ and the involution are
  independent of $\varphi$, and hence this involution is the identity
  map.
\end{Remark}

The key property, \cite[Thm.\ 1.3]{Se}, of the equivariant
pair-of-pants product map $\mathlarger{\wp}$ is that when, for
example, $(M,\omega)$ is symplectically aspherical it becomes an
isomorphism once $\hh^{-1}$ is attached to the ground ring, i.e.,
after taking tensor product with the ring of Laurent series
$\F_2((\hh))$. This yields the Floer theoretic analogue of Borel's
localization relating the filtered cohomology $\HF^*(\varphi)$ and
$\HF^*_{\eq}\big(\varphi^2\big)$ and, as a consequence, a variant of
Smith's inequality, cf.\ \cite{CG, He, Sh, ShZ}. Although in this
paper we do not directly use any of these results, we will briefly
revisit them in Remark \ref{rmk:isomorphism}.

Next, consider the map
$\CS\colon\CF^*(\varphi)\to \CF^*(\varphi)\otimes_{\F_2}
\CF^*(\varphi)$ given by $c\mapsto c\otimes c$ for all
$c\in\CF^*(\varphi)$. When needed, we extend this map to
$\CF^*(\varphi)[[\hh]]$ by setting $\CS(\hh c)=\hh \CS(c)$. Note that,
since we tensor over $\F_2$, the map $\CS$ is not homogeneous in
$\qq$. In general, $\CS$ is neither linear (even over $\F_2$) nor,
when linear, is it a chain map. However, $\CS$ is well-defined on the
level of cohomology and becomes $\F_2[[\hh]]$-linear when multiplied
by $\hh$, i.e., as a map
\begin{equation}
  \label{eq:square}
  \hh \,\CS\colon \HF^*(\varphi)[[\hh]] \to
  \H^*\big(\Z_2; \CF^*(\varphi)
  \otimes_{\F_2}  
  \CF^*(\varphi)\big).
\end{equation}
For instance, to see the linearity it suffices to observe that
\[
  \hh\big((c_1+c_2)^2-c_1^2-c_2^2\big)=d_{\Z_2}(c_1\otimes c_2),
\]
when $d_{\Fl}(c_1)=0=d_{\Fl}(c_2)$. As a consequence, $\CS$ itself is
defined and $\F_2[[\hh]]$-linear on the level of cohomology when the
target has no $\hh$-torsion. Then, composing the cohomology map $\CS$
with the equivariant pair-of-pants map $\mathlarger{\wp}$, we obtain a
map
\[
\PS\colon \HF^*(\varphi)[[\hh]]\to \HF_{\eq}^*\big(\varphi^2\big);
\]
the notation is borrowed from \cite{Wi1, Wi2}.  This map is
$\F_2[[\hh]]$-linear whenever the target has no $\hh$-torsion and, as
we will soon see, this condition is automatically satisfied in the
case we are interested in.

\begin{Remark}
\label{rmk:alg}
  Following \cite[Sect.\ 2.1]{Se} note that for purely
  algebraic reasons there is a canonical isomorphism
  \[
  \H^*\big(\Z_2; \CF^*(\varphi)
  \otimes_{\F_2}\CF^*(\varphi)\big)\cong
  \H^*\big(\Z_2; \HF^*(\varphi) \otimes_{\F_2}
  \HF^*(\varphi)\big)
  \]
  Hence the cohomology group on the right can also be thought of as
  the domain of the equivariant pair-of-pants product
  $\mathlarger{\wp}$ and the target of the map $\CS$. Furthermore, the
  map
  \[
  \CS\colon \HF^*(\varphi)((\hh))\to\H^*\big(\Z_2; \CF^*(\varphi)
  \otimes_{\F_2}
  \CF^*(\varphi)\big)\otimes_{\F_2 [[\hh]]}
  \F_2((\hh)),
  \]
  which is $\F_2((\hh))$-linear regardless of whether or not the
  target of $\hh\,\CS$ in \eqref{eq:square} has $\hh$-torsion, is also
  an isomorphism, again for purely algebraic reasons.
\end{Remark}

In general, we still have $\PS$ defined on the chain level as a map
$$
\CF^*(\varphi)[[\hh]]\to \CF^*_{\eq}\big(\varphi^2\big)
$$
such that $\PS(c)= c * c+ O(\hh)$ for $c\in\CF^*(\varphi)$, but it is
neither $\F_2$-linear nor a chain homomorphism. Recall that the map
$\CS$ itself is not homogeneous in $\qq$. However the composition
$\PS$ is homogeneous, since $\wp$ satisfies
$\wp(\qq c_1\otimes c_2)=\wp(c_1 \otimes \qq c_2)$. (Here we have once
again ignored the shift of degree: by construction,
$|\PS(c)|=2|c|+n$.)

By construction the equivariant pair-of-pants map $\mathlarger{\wp}$
and the map $\PS$ preserve the action filtration; cf.\ Remark
\ref{rmk:reg}.

One of the key ingredients in the proof of \cite[Thm.\ 1.3]{Se} is the
following result, which also plays a central role in our argument and
which, slightly deviating from \cite{Se}, we state for the map $\PS$
rather than for $\mathlarger{\wp}$.

\begin{Proposition}[\cite{Se}, Prop.\ 6.7]
  \label{prop:PS}
  Consider a collection of orbits $\bx_i\in \bPP_1(\varphi)$,
  $i=1,\ldots, \ell$, such that $\CA_H(\bx_i)=a$ for
  $i=1,\ldots, \ell_0\leq \ell$ and $\CA_H(\bx_i)>a$ for the remaining
  orbits. Then
  $$
  \PS\colon
  \sum_{i=1}^\ell  \bx_i\mapsto \sum_{i=1}^{\ell_0}
  \hh^{m_i}\bx^2_i+\ldots ,
  $$
  where $\bx^2_i\in \bPP_2(\varphi)$ is the second iterate of $\bx_i$
  and
  \begin{equation}
    \label{eq:h-deg}
    m_i=2\mu(\bx_i)-\mu\big(\bx^2_i\big)+n   
  \end{equation} 
  and the dots stand for a sum of capped orbits with action strictly
  greater than $2a$.
\end{Proposition}

For instance, $\PS(\bx)=\hh^m\bx^2+\ldots$, where
$m=2\mu(\bx)-\mu(\bx^2)+n$ and the remaining terms have action higher
than that of $\bx^2$. In particular, $\bx^2$ with some power of $\hh$
is necessarily present in $\PS(\bx)$.

Proposition \ref{prop:PS} is an equivariant analogue of the standard
fact that a constant solution of the Floer or Cauchy--Riemann equation
is automatically regular whenever the relative index of the solution
is zero, which in turn is a consequence of that the kernel of the
linearized operator at the constant solution with suitable boundary
conditions is trivial; see, e.g., \cite[Lemma 3.1]{CGG}, \cite[Lemma
6.7.6]{MS}, \cite[p.\ 971]{Se} and \cite[Sect.\ 2.7]{Sa}. However, the
step from a non-equivariant to equivariant setting is non-trivial. We
refer the reader to \cite{Se} for the proof; see also \cite{ShZ} for
an alternative approach and generalizations.

\begin{Remark}[Borel's localization theorem according to \cite{Se}]
  \label{rmk:isomorphism}
  As has been mentioned above, one consequence of Proposition
  \ref{prop:PS} is \cite[Thm.\ 1.3]{Se} asserting, in particular, that
  for a symplectically aspherical manifold the equivariant
  pair-of-pants product $\mathlarger{\wp}$ in the filtered Floer
  cohomology (i.e., the homology of the Floer complex restricted to an
  action interval) becomes an isomorphism after tensoring with
  $\F_2((\hh))$. (The theorem follows from the proposition via
  applying the spectral sequence comparison theorem to the action
  filtration spectral sequence.) Since $\CS$ is an isomorphism modulo
  $\hh$-torsion for purely algebraic reasons (see Remark
  \ref{rmk:alg}), this yields that the map $\PS$ is also an
  isomorphism, as well as the variants of Borel's localization and
  Smith's inequality in the filtered Floer cohomology.

  On the other hand, while for a closed symplectic manifold $M$ the
  analogues of Borel's localization and Smith's inequality hold
  trivially in the global Floer cohomology, the filtered version of
  Borel's localization (in the most naive form) fails without the
  assumption that $(M,\omega)$ is symplectically aspherical; see,
  however, \cite{Sh:HZ}. Moreover, $\PS$ need not be an isomorphism
  even globally without this assumption. In fact, as Example
  \ref{ex:S2} shows, $\PS$ is not an epimorphism already for $M=S^2$.
\end{Remark}  

\begin{Remark}[Regularity]
  \label{rmk:reg}
  To ensure that the regularity condition is satisfied for the
  equivariant pair-of-pants product, a certain arbitrarily small
  ``inhomogeneous perturbation'', i.e., an $s$-dependent perturbation
  of the Hamiltonian, is introduced to the Floer equation in
  \cite{Se}. This perturbation is compactly supported in $s$ and thus
  does not affect the initial and terminal Hamiltonians and the
  actions. However, it does affect the relation between the energy of
  a pair-of-pants curve and the action difference. As a consequence,
  the equivariant pair-of-pants product $\mathlarger{\wp}$ is now
  action increasing only up to an $\eps$-error, which goes to zero
  with the size of the perturbation.  Therefore, $\PS$ also preserves
  the action filtration only up to an $\eps$-error. In particular, if
  all fixed points of $\varphi$ have distinct actions and $\eps>0$ is
  small, $\PS$ literally preserves the action filtration. This would
  already be sufficient for our purposes; see Remark
  \ref{rmk:pert}. However, since the pair-of-pants curves connecting
  different orbits must have energy \emph{a priori} bounded away from
  zero (cf.\ \cite[Prop.\ 2.2]{GG:Rev}), the map $\PS$ always
  preserves the action filtration when the inhomogeneous perturbation
  is small enough, and Proposition \ref{prop:PS} holds as stated.
\end{Remark}

\subsection{Quantum Steenrod square}
\label{sec:QS}
The counterpart of the map $\PS$ on the side of the quantum cohomology
is the \emph{quantum Steenrod square} $\QS$. This quantum cohomology
operation is studied in detail in \cite{Wi1, Wi2}, but the first
Morse/Floer theoretic descriptions of the Steenrod squares go back to
\cite{Be, BC, Fu}. Throughout this section it is essential that the
manifold $M$ is closed.

Following \cite{Se, Wi1, Wi2} and slightly changing the usual
notation, let us define the \emph{Steenrod square} as the degree
doubling linear map
\begin{equation}
  \label{eq:sq1}
\Sq\colon \H^*(M)\to \H^*(M)[[\hh]], \quad
\Sq(\alpha)=\sum_{i=0}^{|\alpha|}
\hh^{|\alpha|-i}\Sq^i(\alpha),
\end{equation}
where $\Sq^i$ are the standard Steenrod squares. In particular,
$|\Sq^i(\alpha)|=|\alpha|+i$, and $\Sq^0(\alpha)=\alpha$ and
$\Sq^{|\alpha|}(\alpha)=\alpha\cup\alpha$. For instance, for the
generator $\varpi$ of $\H^{2n}(M^{2n})$,
\begin{equation}
  \label{eq:Sq(M)}
\Sq(\varpi)=\hh^{2n}\varpi.
\end{equation}

The \emph{quantum Steenrod square} is a degree doubling map
$$
\QS\colon \H^*(M)\to \HQ^*(M)[[\hh]],
$$
which is a certain deformation of $\Sq$ in $\qq$:
\begin{equation}
    \label{eq:QS(M)}
\QS(\alpha)=\Sq(\alpha)+O(\qq)
\end{equation}
for $\alpha\in \H^*(M)$. For instance,
$$
\QS(\varpi)=\hh^{2n}\varpi+O(\qq),
$$
and $\QS$ is undeformed at $\varpi\in \H^{2n}(M)$ if and only if the
higher order terms in $\qq$ vanish.

It is convenient to formally extend $\Sq$ and $\QS$ to the maps
$$
  \Sq\colon \HQ^*(M)[[\hh]]
  \to
  \HQ^*(M)[[\hh]]
$$
and
$$
\QS\colon \HQ^*(M)[[\hh]]\to \HQ^*(M)[[\hh]],
$$
which are linear over $\F_2[[h]]$ and homogeneous of degree two in
$\qq$. More precisely, for instance, we set
$\QS(\hh\alpha)=\hh\QS(\alpha)$ and
$\QS(\qq \alpha)=\qq^2 \QS(\alpha)$. Since the extension is linear
over $\F_2[[h]]$, the maps are no longer degree doubling. In what
follows, unless stated otherwise, $\Sq$ and $\QS$ we will refer to the
extended maps above. Note that $\QS$ is still a deformation of $\Sq$
in $\qq$ in the sense of \eqref{eq:QS(M)}.

The next ingredient we need is the equivariant continuation/PSS map
introduced in \cite{Wi2}. This is the map $\Phi_{\eq}$ from the
$\Z_2$-equivariant cohomology of $\HQ^*(M)$ with trivial action to the
equivariant Floer cohomology of $\varphi^2$:
$$
\Phi_{\eq}\colon \HQ^*_{\eq}(M)\to \HF^*_{\eq}\big(\varphi^2\big)[n].
$$
Since the $\Z_2$-action in the cohomology is trivial, we have
$$
\HQ^*_{\eq}(M)=\HQ^*(M)[[\hh]]
$$
and one can also think of the cohomology on the left hand side as the
target space of $\QS$.

Just as an ordinary continuation/PSS map $\Phi$, its equivariant
counterpart $\Phi_{\eq}$ is an $\F_2[[\hh]]$-linear isomorphism. As a
consequence, $\HF^*_{\eq}\big(\varphi^2\big)$ has no $\hh$-torsion and
the map $\PS$ is linear.

The spaces and maps we have introduced fit together into the following
commutative diagram, where we again suppressed the shifts of degree by
the continuation/PSS maps:
\begin{equation}
\label{eq:diag1}
  \xymatrix
  {
\HQ^*(M)[[\hh]]\ar[d]_{{\QS}} \ar[r]^{{\Phi}}_{{\cong}}
&
\HF^*(\varphi)[[\hh]] \ar[d]^{\PS}
\\
\HQ^*_{\eq}(M) \ar[d]_{\cong} \ar[r]^{\Phi_{\eq}}_{{\cong}}
&
\HF_{\eq}^*\big(\varphi^2\big) 
\ar[d]^{F}
\\
\HQ^*(M)[[\hh]] \ar[r]^{\Phi}_{{\cong}}
&
\HF^*\big(\varphi^2\big)[[\hh]]
}
\end{equation}
We emphasize that here the continuation/PSS maps $\Phi$ for $\varphi$
and $\varphi^2$, and $\Phi_{\eq}$ (i.e., the horizontal arrows) are
$\F_2[[\hh]]$-linear isomorphisms. The requirement that the top square
is commutative can be viewed as the definition of $\QS$,
\cite{Wi2}. Likewise, the condition that the bottom square is
commutative is the definition of $F$, i.e.,
$$
F=\Phi\Phi^{-1}_{\eq}.
$$
It is worth pointing out that in general the map $F$ need not preserve
the action filtration.

\begin{Example}
  \label{ex:S2}
  Let $M=S^2$. Then $\QS(1)=1$ and
  $\QS(\varpi)=\hh^2\varpi+\qq\cdot 1$, where $1$ is the generator of
  $\H^0(S^2)$; see \cite{Wi1}. Here $\qq\cdot 1$ is the deformation
  term, which is equal to the quantum square of $\varpi$. Continuing
  the discussion from Remark \ref{rmk:isomorphism} and making $\hh$
  invertible, it is now easy to see that neither $\varpi$ nor
  $\qq\cdot 1$ nor $\qq \varpi$ is in the image of
  \[
    \QS\colon \HQ^*(S^2)((\hh))\to\HQ^*_{\eq}(S^2)
    \otimes_{\F_2[[\hh]]}\F_2((\hh))\cong \HQ^*(S^2)
    \otimes_{\Lambda}\Lambda((\hh)).
  \]
  As a consequence of the diagram \eqref{eq:diag1},
  \[
    \PS\colon \HF^*(\varphi)((\hh))\to
    \HF^*_{\eq}(\varphi^2)\otimes_{\F_2[[\hh]]}\F_2((\hh))
  \]
  is not onto for any Hamiltonian diffeomorphism
  $\varphi\colon S^2\to S^2$. Note that at the same time linearly
  extending the classical Steenrod square map \eqref{eq:sq1} over
  $\F_2((\hh))$ gives a linear isomorphism
  $\H^*(M)((\hh))\to \H^*(M)((\hh))$ for any closed manifold $M$.
\end{Example}
    
Finally, as has been mentioned above, the complex
$\CF^*_{\eq}\big(\varphi^2\big)$ carries the $\hh$-adic
filtration:
  $$
  \CF^*_{\eq}\big(\varphi^2\big)\supset \hh
  \CF^*_{\eq}\big(\varphi^2\big)
  \supset \hh^2 \CF^*_{\eq}\big(\varphi^2\big)\supset \ldots .
  $$
  The associated spectral sequence (the Leray spectral sequence in the
  equivariant cohomology) converges to the graded space $E_\infty$
  associated with the $\hh$-adic filtration of
  $\HF^*_{\eq}\big(\varphi^2\big)$. It readily follows from the fact
  that $\Phi_{\eq}$ is an isomorphism that this spectral sequence
  collapses on the $E_1$-page:
  $E_1=\HF^*\big(\varphi^2\big)[[\hh]]=E_\infty$. Indeed, the fact
  that $E_1$ in every bi-degree has the same dimension over $\F_2$ as
  $E_\infty$ forces all higher order differentials to
  vanish. Alternatively, we can view $\CF_{\eq}^*\big(\varphi^2\big)$
  as an ungraded complex over $\Lambda$ with $\hh$-adic
  filtration. Then $E_1$ is finite-dimensional over $\Lambda$ in every
  degree with $\dim_\Lambda E_1^*=\dim_\Lambda E_\infty^*$, which
  again implies that the spectral sequence collapses in the
  $E_1$-term. With this in mind we can view
  $\HF^*\big(\varphi^2\big)[[\hh]]$ as the graded space associated
  with the $\hh$-adic filtration on $\HF^*_{\eq}\big(\varphi^2\big)$.

  \begin{Remark}
    \label{rmk:graded}
    It is essential to acknowledge an abuse of terminology in the
    paragraph above and in what follows. (We are grateful to the
    referee for pointing it out.) Strictly speaking, since the
    $\hh$-adic filtration is infinite, to keep the isomorphism between
    $E_\infty$ and $\HF_{\eq}^*\big(\varphi^2\big)$ we would need to
    define the $E_\infty$-page of the spectral sequence as the direct
    product of individual degree terms rather than the direct
    sum. Then $E_\infty$ becomes a filtered (rather than graded)
    module which is isomorphic to $\HF_{\eq}^*\big(\varphi^2\big)$ as
    a filtered module, but not to the graded module resulting from the
    $\hh$-adic filtration of $\HF_{\eq}^*\big(\varphi^2\big)$. The
    problem here stems from the difference between the direct product
    and the direct sum; e.g., $\F_2[[h]]$ is the product of its fixed
    degree terms and not a graded ring, whereas $\F_2[h]$ is the
    direct sum and graded. Below we will ignore this issue, which is
    not uncommon in symplectic topology literature, for it is
    essentially of terminological nature. Furthermore, the problem can
    be completely avoided as explained in the next remark.
  \end{Remark}
  
  \begin{Remark}
\label{rmk:CFeq}
In \eqref{eq:CFeq}, we defined $\CF_{\eq}^*\big(\varphi^2\big)$ as the
space of formal power series with coefficients in
$\CF^*\big(\varphi^2\big)$. There are, however, other choices of
$\CF_{\eq}^*\big(\varphi^2\big)$ although the difference is purely
technical, cf.\ \cite{Se, Wi1, Wi2}. For instance, by examining the
action/index change one can readily see that for a closed monotone
symplectic manifold, the expansion of $d_{\eq}$ involves only a finite
number of non-zero terms. Thus the equivariant Floer homology can be
defined over $\F_2[\hh]$ as $\CF^*\big(\varphi^2\big)[\hh]$. In a
similar vein, the product $\PS$, the quantum Steenrod square $\QS$,
and the continuation/PSS maps $\Phi_{\eq}$ and $\Phi_{\eq}^{-1}$ are
also polynomial in $\hh$. (Note that as a consequence, $F$ is also
polynomial in $\hh$, which is \emph{a priori} non-obvious.)
Therefore, in all constructions throughout the paper one can replace
formal power series in $\hh$ by polynomials, avoiding the
terminological issue pointed out in Remark \ref{rmk:graded}.  Finally,
one can choose a middle way and replace the space of formal power
series by the tensor product with $\F_2[[\hh]]$ over $\F_2$, e.g.,
setting
$\CF_{\eq}^*\big(\varphi^2\big) =
\CF^*\big(\varphi^2\big)\otimes_{\F_2}\F_2[[\hh]]$.
  \end{Remark}

We will need the following simple observation:

\begin{Lemma}
  \label{lemma:Phi}
  For any Hamiltonian diffeomorphism $\varphi$, the equivariant
  continuation map $\Phi_{\eq}$ induces the ordinary continuation map
  $\Phi$ on the graded vector space
  $E_\infty=\HF^*\big(\varphi^2\big)[[\hh]]$, i.e.,
\begin{equation}
  \label{eq:Phi}
\Phi_{\eq}=\Phi+O(h).
\end{equation}
\end{Lemma}

\begin{proof}
  Let $f$ be a $C^2$-small Morse function on $(M, \omega)$, unrelated
  to the pseudo-rotation $\varphi$. Set
  $\QC^*(M) = \CMo^*(f) \otimes \Lambda$ where $\CMo^*(f)$ is the
  Morse complex of $f$. The complex $\QC^*(M)[[\hh]]$ has a natural
  $\hh$-adic filtration and the resulting spectral sequence collapses
  on the $E_1$-page; for $\hh$ is not involved in the
  differential. Furthermore, recall that
  $\HF^*\big(\varphi\big)[[\hh]]$ is the $E_1$-page associated with
  the $\hh$-adic filtration of $\CF^*_{\eq}\big(\varphi^2\big)$. We
  claim that, in the obvious notation,
\begin{equation}
  \label{eq:E1}
E_1(\Phi_{\eq})=\Phi.
\end{equation}
Then, since both spectral sequences collapse on the $E_1$-page,
\eqref{eq:Phi} readily follows from \eqref{eq:E1}.

It remains to establish \eqref{eq:E1}.  Let $\psi_f$ be the
Hamiltonian diffeomorphism generated by $f$. Following \cite{Wi2}, we
write the chain level definitions of $\Phi_{\eq}$ and $\Phi$ as the
compositions
\[
\QC^*(M) [[\hh]] \stackrel{\Psi_{\eq}}{\longrightarrow}
\CF^*\big(\psi^2_f\big)[[\hh]]  \stackrel{C_{\eq}}{\longrightarrow}
\CF^*\big(\varphi^2\big)[[\hh]]=\CF^*_{\eq}\big(\varphi^2\big)
\]
and
\[
\QC^*(M) [[\hh]]  \stackrel{\Psi}{\longrightarrow}
\CF^*\big(\psi^2_f\big)[[\hh]]  \stackrel{C}{\longrightarrow}
\CF^*\big(\varphi^2\big)[[\hh]].
\]
Here the map $\Psi$ is the PSS map for $f$ or, to be more precise, for
$2f$. Furthermore, $C$ is the continuation from the Floer complex of
$\varphi^2$ to the Floer complex of $\psi_f^2$ for a fixed (e.g.,
linear) homotopy between $f$ and the Hamiltonian $H$ generating
$\varphi$. The maps $\Psi_{\eq}$ and $C_{\eq}$ are the equivariant
counterparts of $\Psi$ and, respectively, $C$. (Note that the
differential in the Morse complex of $f$, and hence in the Floer
complex of $\psi_f^2$, might be non-trivial; for $M$ need not admit a
perfect Morse function.)  By definition (see \cite{Wi2}),
$\Psi_{eq}=\Psi + O(\hh)$ and $C_{\eq}=C+O(\hh)$, and \eqref{eq:E1}
follows.
\end{proof}

\subsection{Enter pseudo-rotations} Assume now that $\varphi$ is a
pseudo-rotation or more generally that every 2-periodic point is a
fixed point and that $d_{\Fl}=0$ for $\varphi$ and hence for
$\varphi^2$. Then $\HF^*(\varphi)=\CF^*(\varphi)$ and
$\HF^*\big(\varphi^2\big)=\CF^*\big(\varphi^2\big)$.

Furthermore, from the collapse of the Leray spectral sequence it then
follows inductively that $d_{\eq}=0$ and we have the identifications
\begin{equation}
  \label{eq:F0}
  F_0 \colon
  \HF^*_{\eq}\big(\varphi^2\big)=\CF^*_{\eq}\big(\varphi^2\big)
  = \CF^*\big(\varphi^2\big)[[\hh]],
\end{equation}
which, in contrast with the natural map $F$, are specific to the case
of pseudo-rotations, but might exist under somewhat less restrictive
conditions. Of course, $F_0=\id$, but we prefer to use a different
notation at this point to emphasize the fact that $F_0$ is defined
only under some additional assumptions on $\varphi$ and $\varphi^2$.

The next important (but simple) ingredient of our proof, which we use
to establish \eqref{eq:bx2l} below, is the following lemma.

\begin{Lemma}
  \label{lemma:F0vsF}
We have $F=F_0+O(\hh)$.
\end{Lemma} 

This lemma is an immediate consequence of Lemma \ref{lemma:Phi} and
the identifications~\eqref{eq:F0}.

\begin{Example}
  \label{Ex:F}
  Returning to Example \ref{ex:S2} consider the rotation $\varphi$ of
  $S^2$ about the $z$-axis in an angle $\theta$. Let $\bx$ and $\by$
  be the South and North poles respectively, equipped with trivial
  cappings. Thus $\mu(\bx)=-1$ and $\mu(\by)=1$. Passing to the second
  iterate $\varphi^2$, we still have $\mu\big(\bx^2\big)=-1$ and
  $\mu\big(\bx^2\big)=1$ when $\theta\in (0,\,\pi)$.  Then
  $\Phi_{\eq}(1)=\bx^2=\Phi(1)$ and
  $\Phi_{\eq}(\varpi)=\by^2=\Phi(\varpi)$ and $F=\id$. In general,
  $F=\id$ whenever $\varphi$ is sufficiently close to $\id$. On the
  other hand, if $\theta\in (\pi,\,2\pi)$ we have
  $\mu\big(\bx^2\big)=-3$ and $\mu\big(\by^2\big)=3$. Then
  $\Phi_{\eq}(1)=\qq^{-1}\by^2+\hh^2\bx^2$, while
  $\Phi(1)= \qq^{-1}\by^2$, and
  $\Phi_{\eq}(\varpi)=\qq\bx^2=\Phi(\varpi)$. (This can be proved by
  using the information about $\QS$ from Example \ref{ex:S2} along
  with index/action analysis and Lemma \ref{lemma:F0vsF}.) As a
  consequence, $F\big(\bx^2\big)=\bx^2$ and
  $F\big(\by^2\big)=\by^2+\hh^2\qq\bx^2$. We are not aware of any
  general method of explicitly calculating the map $F$.
\end{Example}

\section{Proof of the main theorem}
\label{sec:proof}
With all preparations in place, we are now ready to prove the main
result of the paper, Theorem \ref{thm:main}.  For the sake of
simplicity, we will assume that all capped periodic orbits have
distinct action: the argument extends to the general case in a
straightforward way and the difference is purely expository. (See also
Remark \ref{rmk:pert}.)

We will argue by contradiction: throughout the proof we assume that
$\QS(\varpi)$ is undeformed, i.e., $\QS(\varpi)=\hh^{2n}\varpi$; see
\eqref{eq:Sq(M)} and \eqref{eq:QS(M)}. The proof comprises two steps,
and this assumption, which we aim to disprove, is used in both steps.

Write
\begin{equation}
  \label{eq:Phi(M)}
\Phi(\varpi)=\bx+\ldots,
\end{equation}
where the dots stand for the terms with action strictly greater than
the action of $\bx$. Thus $\mu(\bx)=n$. In the first step we prove the
theorem under the additional condition that the index of $\bx$ jumps
from $\bx$ to $\bx^2$, i.e.,
\begin{equation}
  \label{eq:index}
  \mu\big(\bx^2\big)>\mu(\bx)=n.
\end{equation}

This condition might or might not hold for $\varphi$ and in the second
step we show that \eqref{eq:index} is necessarily satisfied for a
sufficiently high iterate of $\varphi$. This will complete the proof
of the theorem. 

\emph{Step 1.}  Thus let us prove the theorem under the additional
condition \eqref{eq:index}, where $\bx$ is given by \eqref{eq:Phi(M)}.
From the top square in the diagram \eqref{eq:diag1} and Proposition
\ref{prop:PS}, we obtain the following commutative square:
\begin{equation*}
\xymatrix
{
  \varpi \ar[d]_{{\QS}} \ar[r]^{{\Phi}} 
  & 
  \bx+ \ldots \ar[d]^{\PS}
  \\
  \hh^{2n}\varpi \ar[r]^{\Phi_{\eq}}
  & 
  \hh^m\bx^2+\ldots\\
 }
\end{equation*}
In the bottom right corner the dots  again stand for higher action
terms and, by \eqref{eq:h-deg} and \eqref{eq:index},
$$
m=3n-\mu\big(\bx^2\big) <2n.
$$
This contradicts the fact that $\Phi_{\eq}$ is $\F_2[[\hh]]$-linear.

\emph{Step 2.} To finish the proof it remains to make sure that
\eqref{eq:index} is satisfied after, if necessary, replacing $\varphi$
by its iterate.

The condition that $\mu(\bx)=n$ guarantees that, by
\eqref{eq:mean-cz}, $\hmu(\bx)>0$ and hence
$\mu\big(\bx^k\big)\to\infty$, and furthermore that the sequence
$\mu\big(\bx^k\big)$ is increasing (but not necessarily strictly
increasing):
\begin{equation}
  \label{eq:increasing}
  \mu\big(\bx^k\big)\nearrow\infty .
\end{equation}

There are several ways to show this. For instance, let us adopt the
argument from \cite[Sect.\ 4]{GG:PRvR}; see also \cite[Formula
(6.1)]{GG:PRvR}. Namely, let $P\in \tSp(2n)$ be the linearized flow
along $\bx$. Since the index sequence $\mu\big(P^k\big)$ is invariant
under iso-spectral deformations, we can assume without loss of
generality that $P(1)$ is semi-simple. Then $P$ can be expressed as
the product of a loop $\phi$ and $P\in\tSp(2n)$ which decomposes as a
direct sum of elements of $\tSp(2)$ or $\tSp(4)$ of the following
three types: short rotations of $\R^2$ (by an angle
$\theta\in (-\pi,\pi)$), positive and complex hyperbolic
transformations of $\R^2$ or $\R^4$ with zero index, and negative
hyperbolic transformations of $\R^2$. (A negative hyperbolic
transformation is the counterclockwise rotation in $\pi$ composed with
a positive hyperbolic transformation with zero index.) Then, using the
condition that $\mu(P)=n$, we can redistribute the loop part $\phi$
among individual terms and write $P$ as $\bigoplus P_i$, where $P_i$
is either a counterclockwise rotation by $\theta\in (0,2\pi)$ or a
counterclockwise negative hyperbolic transformation. Clearly, each of
the sequences $\mu\big(P_i^k\big)$ is increasing, and hence so is
$\mu\big(P^k\big)$.

As a consequence of \eqref{eq:increasing}, there exists
$r=2^{\ell_0}\geq 1$ such that $\mu\big(\bx^k\big)=n$ for $k\leq r$
but $\mu\big(\bx^{2r}\big)>n$.

We claim that
\begin{equation}
  \label{eq:bx2l}
\Phi(\varpi)=\bx^{2^\ell}+\ldots
\end{equation}
as long as $\ell\leq \ell_0$, where the dots stand again for the terms
of higher action. In particular, we can replace $\varphi$ by
$\varphi^r$ to guarantee that \eqref{eq:index} is satisfied.

To prove \eqref{eq:bx2l}, arguing by induction, it is enough to show
that \eqref{eq:Phi(M)} still holds for $\varphi^2$, i.e.,
\begin{equation}
  \label{eq:bx2}
\Phi(\varpi)=\bx^2+\ldots, 
\end{equation}
provided, of course, that it holds for $\varphi$ and
$\mu\big(\bx^2\big)=n=\mu(\bx)$.

To establish \eqref{eq:bx2}, let us trace the image of $\varpi$
through the diagram \eqref{eq:diag1}. We have
$$
  \xymatrix
  {
\varpi\ar[d]_{\QS}\ar[r]^{\Phi} 
&
\bx+\ldots \ar[d]^{\PS}
\\
\hh^{2n}\varpi \ar[d]_{\cong} \ar[r]^{\Phi_{\eq}}
&
\hh^{2n}\big(\bx^2+R\big)
\ar[d]^{F}
\\
\hh^{2n}\varpi\ar[r]^{\Phi}
&
\hh^{2n}\big(\bx^2+R'\big).
}
$$
We emphasize that in the left column of the diagram, we have used, as
in Step 1, the background assumption that $\QS(\varpi)$ is
undeformed. Next, let us take a closer look at what $R$ and $R'$ are.

The remainder $R$ in $\PS(\bx+\ldots)$ is a sum of capped 2-periodic
orbits of $\varphi$ with action strictly greater than the action of
$\bx^2$ and possibly $\hh$-dependent coefficients. The condition that
$$
\PS(\bx+\ldots)=\Phi_{\eq}\big(\hh^{2n}\varpi\big)
=\hh^{2n}\Phi_{\eq}(\varpi)
$$
guarantees that $\PS(\bx+\ldots)$ is divisible by $\hh^{2n}$. Let us
write
$$
R=\sum\bz_j+O(\hh),
$$
where $\bz_j$ are some capped 2-periodic orbits of $\varphi$ with
action strictly greater than the action of $\bx^2$.

By Lemma \ref{lemma:F0vsF}, $F\big(\bx^2\big)=\bx^2+O(\hh)$ and 
$$
R'=\sum \bz_j+O(\hh).
$$
However, $\Phi$ is a non-equivariant continuation/PSS map and thus
$\Phi(\varpi)\in\CF^*\big(\varphi^2\big)$. Therefore,
$R'\in\CF^*\big(\varphi^2\big)$ since
$$
\hh^{2n}\big(\bx^2+R'\big)=\hh^{2n}\Phi(\varpi).
$$
This proves \eqref{eq:bx2} and completes the proof of the
theorem. \hfill$\qed$

\begin{Remark}
  \label{rmk:pert}
  Returning to the regularity question (see Remark \ref{rmk:reg}),
  note that, the general case of the theorem can also be reduced to
  the case considered above where all periodic orbits of $\varphi$
  have distinct action. Indeed, it is clear from the proof that it
  suffices to assume that $\varphi$ is a pseudo-rotation up to a
  certain iteration order $r$, which is completely determined by the
  indices and the mean indices of 1-periodic orbits. Then the actions
  can be made distinct by an arbitrarily small perturbation of
  $\varphi$ keeping it a pseudo-rotation up to arbitrarily large
  iteration order. (A somewhat similar type of perturbation is used,
  for instance, in \cite[Sect.\ 3.2]{GG:Rev} in the proof of a Conley
  conjecture type result.)
\end{Remark}

\end{document}